\documentclass[14pt]{amsart}
\address{\newline{\normalsize Courant Institute, NYU, 251 Mercer str., New York, NY 10012, USA}\newline{\it E-mail address}:
karzhema@cims.nyu.edu}
\usepackage{amscd,amssymb}
\usepackage{amsthm,amsmath,amssymb}
\usepackage[matrix,arrow]{xy}

\makeatletter\@addtoreset{equation}{section}\makeatother

\makeatletter\@addtoreset{subsection}{equation}\makeatother

\newtheorem{theorem}[equation]{Theorem}
\newtheorem{proposition}[equation]{Proposition}
\newtheorem{lemma}[equation]{Lemma}

\newtheorem{theorem-definition}[equation]{Theorem-definition}

\theoremstyle{definition}

\newtheorem*{notation}{Notation}

\theoremstyle{remark}
\newtheorem{remark}[equation]{Remark}

\begin{document}



\title{On the quotient of $\mathbb{C}^4$ by a finite primitive group of type (I)}


\author{Ilya Karzhemanov}

\begin{abstract}
We study rationality problem for the quotient of $\mathbb{C}^4$ by
a finite primitive group $G$ of Type (I). We prove that this
quotient is a rational variety for any such $G$.
\end{abstract}

\sloppy

\maketitle

\bigskip

\section{Introduction}
\label{section:k-quotient-0}

Given a complex affine space $\mathbb{C}^n =
\mathrm{Spec}(\mathbb{C}[x_{1}, \ldots, x_{n}])$ and a finite
group $G$ acting linearly on $\mathbb{C}^n$, one of the
fundamental questions to ask is whether the field of $G$-invariant
rational functions on $\mathbb{C}^n$ is also a purely
transcendental extension of $\mathbb{C}$, or, in other words,
whether variety $\mathbb{C}^{n}/G$ is rational (see
\cite{prok-inv} (and references therein) for an extensive overview
of the current state of the problem). By a simple argument (see
\cite[Proposition 1.2]{prok-inv}), one can show that
$\mathbb{C}^{n}/G$ is birationally isomorphic to
$(\mathbb{P}(\mathbb{C}^{n})/G) \times \mathbb{P}^1$, and hence $n
= 4$ is the first non-trivial issue, since the L\"uroth problem
has a positive solution for $n \leqslant 3$. The case of $n = 4$
has been treated in detail in \cite{prok-inv}. However, for some
of the groups $G$ (non-)rationality of $\mathbb{C}^{4}/G$ was not
established.

Namely, let $\mathbb{O}, \mathbb{I} \subset SL_{2}(\mathbb{C})$ be
the octahedron and icosahedron subgroups, respectively. Identify
$U_0 := \mathbb{C}^4$ with the
space of $(2 \times 2)$-matrices $A := \begin{pmatrix} X_1 & X_2 \\
X_3 & X_4 \end{pmatrix}$, $X_i \in \mathbb{C}$, and consider the
action of the group $G := \mathbb{O} \times \mathbb{I}$ on $U_0$
such that $\mathbb{O}$ and $\mathbb{I}$ act by multiplying $A$
from the left and right, respectively. Furthermore, by the above
argument in order to establish rationality of $U_{0}/G$, one may
assume that $G := (\mathbb{O} \times \mathbb{I}) \cdot
\mathbb{C}^*$ for the standard diagonal action of $\mathbb{C}^*$
on $U_0$. Then for such group action we prove the following:

\begin{theorem}
\label{theorem:main} The $3$-fold $U_{0}/G$ is rational.
\end{theorem}

Theorem~\ref{theorem:main} settles the remaining case in
\cite{prok-inv} of quotients of $\mathbb{P}^3$ (or, equivalently,
$\mathbb{C}^4$) by \emph{finite primitive groups of Type (I)} (see
\cite[Section 2]{prok-inv} for the description of these).

Let us outline the proof of Theorem~\ref{theorem:main}. Recall
that in \cite{prok-inv}, after taking the $\mathbb{C}^*$-quotient
of $U_0$ and passing to the projectivized $G$-action on
$\mathbb{P}^3$, with $G$ now equal $\mathbb{O} \times \mathbb{I}$,
one can notice that $\mathbb{P}^3/G$ is birationally isomorphic to
$SL_{2}(\mathbb{C})/G$ for the induced $G$-action on
$SL_{2}(\mathbb{C}) \subset U_0$. Further, compactifying
$SL_{2}(\mathbb{C})$ by a smooth Fano $3$-fold $W$ with either
$\mathbb{O}$- or $\mathbb{I}$-action, one might try to prove that
the corresponding quotient of $W$ is rational by finding an
equivariant birational map of $W$ onto a product of
positive-dimensional varieties (see \cite[Section 2]{prok-inv},
where this idea worked perfectly well for all finite primitive
groups of Type (I), except for the given $G$).

Our approach is more direct (and simpler in a sense). Namely, let
the group $\mathbb{Z}/2\mathbb{Z}$ act on $U_0$ by multiplying
every $X_i$ by $-1$, so that the $G$-action descends to
$U_{0}/(\mathbb{Z}/2\mathbb{Z})$. A natural generalization of the
construction of $\mathbb{P}^1$ leads to a projective
compactification $V'$ of $U_{0}/(\mathbb{Z}/2\mathbb{Z})$ (see
Section~\ref{section:quotient-1} below).\footnote{By ``$V'$
compactifies $U_{0}/(\mathbb{Z}/2\mathbb{Z})$'' we mean that
$\mathbb{C}(V') = \mathbb{C}(U_{0}/(\mathbb{Z}/2\mathbb{Z}))$ for
the fields of meromorphic functions.} This $V'$ turns out to be a
Fano $4$-fold with isolated terminal singularities, of Picard
number $1$ and Fano index $4$, i.e., $V'$ is a quadratic cone in
$\mathbb{P}^5$ by a result of T. Fujita (see
Lemma~\ref{theorem:g-2-4-a-s-2}). Furthermore, the $G$-action on
$U_{0}/(\mathbb{Z}/2\mathbb{Z})$ extends to a regular action on
$V'$, and $V' \subset \mathbb{P}^5$ happens to have three linearly
independent $G$-invariant hyperplane sections (see
Lemma~\ref{theorem:many-invariants}). Then, considering the
corresponding $G$-equivariant linear projection $V \dashrightarrow
\mathbb{P}^2$, we split the threefold $V'/G$ birationally into a
product of positive-dimensional varieties, thus proving
rationality of $V'/G$ (see Lemma~\ref{theorem:v-prime}). It is now
easy to see that $U_0/G$ is also rational (see
Lemma~\ref{theorem:v-prime-1}).

\begin{remark}
\label{remark:other-proof} Instead of $\mathbb{O} \times
\mathbb{I}$ one may take any other finite primitive group $G$ of
Type (I) and prove that the corresponding quotient
$\mathbb{C}^{4}/G$ is rational, repeating literally the arguments
in Sections~\ref{section:quotient-1} and \ref{section:quotient-2}
below. This gives another proof of Theorem 2.1 in \cite{prok-inv}.
\end{remark}

\begin{notation}
We use standard notions and facts from \cite{isk-prok}. Also
throughout the paper we use the following notation:

\begin{itemize}

\item[(\dag)] Given two varieties $X$ and $Y$, $X
\approx Y$ denotes birational equivalence between them. For an
algebraic group $G$ acting regularly on both $X$ and $Y$, we write
$X \approx_G Y$ if there exists a $G$-equivariant birational map
$X \dashrightarrow Y$.

\end{itemize}

\end{notation}

\bigskip

\section{One explicit compactification}
\label{section:quotient-1}

\refstepcounter{equation}
\subsection{}
\label{subsection:grass-g-2-4}

Take another copy $U_1$ of $\mathbb{C}^4$. Identify $U_1$ with the
space of $(2 \times 2)$-matrices, as $U_0$ above. Let $\varphi_1 :
U_0 \dashrightarrow U_1$ be birational map induced by the morphism
$GL_{2}(\mathbb{C}) \longrightarrow  GL_{2}(\mathbb{C})$ which
sends every invertible matrix $A \in U_0$ to $A^{-1} \in U_1$. Set
$X^{(1)}_{i} := \varphi_{1}^{-1*}(X_{i})$, $1 \leqslant i
\leqslant 4$. These extend to affine coordinates on $U_1$. Put
also $\Delta_0 := \det A$ and $\Delta_1 :=
\varphi_{1}^{-1*}(\Delta_{0})$.

Further, let $l_{\alpha, \beta}$ be the linear automorphism of
$U_0$ which permutes $X_{\alpha}$ and $X_{\beta}$ in $A$ with
$\alpha + \beta \ne 5$. Take another copy $U_{\alpha, \beta}$ of
$\mathbb{C}^4$, as $U_0$ and $U_1$ above, and consider birational
map $\varphi_{\alpha, \beta} := \varphi_1 \circ l_{\alpha, \beta}
: U_0 \dashrightarrow U_{\alpha, \beta}$. Set $X^{(\alpha,
\beta)}_{i} := \varphi_{\alpha, \beta}^{-1*}(X_{i})$. These extend
to affine coordinates on $U_{\alpha, \beta}$. Put also
$\Delta_{\alpha, \beta} := \varphi_{\alpha,
\beta}^{-1*}(\Delta_{0})$.

Now glue $U_0$, $U_1$, $U_{\alpha, \beta}$ together via the maps
$\varphi_{1}$, $\varphi_{\alpha, \beta}$ for various $\alpha,
\beta$. We get a smooth complex $4$-fold $V$ so that $U_0$, $U_1$,
$U_{\alpha, \beta}$ are analytic domains covering $V$. Note that
$\Delta_1 = \Delta_{0}^{-1}$ on $U_0 \cap U_1$ and
$\Delta_{\alpha, \beta} = l_{\alpha, \beta}^{*}(\Delta_{0})$ on
$U_0 \cap U_{\alpha, \beta}$.

\begin{lemma}
\label{theorem:g-2-4} $V = G(2,4)$, the Grassmanian of $2$-planes
in $\mathbb{C}^4$.
\end{lemma}

\begin{proof}
Evident (by definition of the complex structure on $G(2,4)$).
\end{proof}

\refstepcounter{equation}
\subsection{}
\label{subsection:grass-g-2-4-alg-space}

Let us now replace each of $U_i$ and $U_{\alpha, \beta}$ in
{\ref{subsection:grass-g-2-4}} by
$\mathbb{C}^4/(\mathbb{Z}/2\mathbb{Z})$, where
$\mathbb{Z}/2\mathbb{Z}$ acts via $X_i \mapsto -X_i$, $1 \leqslant
i \leqslant 4$. Note that the gluing maps $\varphi_1$ and
$\varphi_{\alpha, \beta}$ are
$(\mathbb{Z}/2\mathbb{Z})$-equivariant, hence we can glue the six
copies of $\mathbb{C}^{4}/(\mathbb{Z}/2\mathbb{Z})$ together via
$\varphi_1$, $\varphi_{\alpha, \beta}$ as above. We get an
algebraic space $V'$ (with $\{U_0, U_1, U_{\alpha,
\beta}\}_{\alpha, \beta}$ being an open cover of $V'$ in the
orbifold topology).

\begin{remark}
\label{remark:al-artin} Note that the gluing maps $\varphi_1,
\varphi_{1, 2}, \ldots$ on $V'$ are rather \emph{algebraic} (see
\cite[Ch.~1]{artin}) than analytic. Indeed, $\varphi_1,
\varphi_{1, 2}$, etc., when lifted to the universal covers of the
charts $U_0 := \mathbb{C}^{4}/(\mathbb{Z}/2\mathbb{Z}),\ldots$,
are only $\mathbb{Z}/2\mathbb{Z}$-equivariant, but not
$\mathbb{Z}/2\mathbb{Z}$-invariant. It is easy to see, however,
that the complex (scheme) structure on $V'$ is provided by the
charts $U_0 \cup U_1, U_0 \cup U_{1, 2}, \ldots$ (but \emph{not}
by $\{U_0, U_1, U_{\alpha, \beta}\}_{\alpha, \beta}$), glued from
$U_0, U_1,U_{1,2}$, etc. via $\varphi_1, \varphi_{1, 2},\ldots$.
\end{remark}

\begin{lemma}
\label{theorem:g-2-4-a-s} $V'$ is compact.
\end{lemma}

\begin{proof}
Let $\Delta\subset \mathbb{C}$ be a small disk around $0$. We have
to prove that any (analytic) family of points $O_t \in V'$,
parameterized by $\Delta\setminus{\{0\}}\ni t$, extends to a
family at $t = 0$. This follows from Lemma~\ref{theorem:g-2-4} and
the fact that the gluing maps $\varphi_1, \varphi_{1, 2}, \ldots$
are $\mathbb{Z}/2\mathbb{Z}$-equivariant.
\end{proof}

The next lemma is straightforward from the construction of $V'$
(cf. Remark~\ref{remark:al-artin}):

\begin{lemma}
\label{theorem:g-2-4-a-s-1} $\mathbb{C}(V') = \mathbb{C}(U_0)$.
\end{lemma}

\begin{remark}
\label{remark:only-bir} One can easily see that the quotient map
$\mathbb{C}^4 \longrightarrow U_0 :=
\mathbb{C}^{4}/(\mathbb{Z}/2\mathbb{Z})$ does not induce a
\emph{regular} map $V = G(2,4) \longrightarrow V'$. Thus, in view
of Lemma~\ref{theorem:g-2-4-a-s-1}, $V'$ is only
\emph{birationally} a quotient $V/(\mathbb{Z}/2\mathbb{Z})$.
\end{remark}

\refstepcounter{equation}
\subsection{}
\label{subsection:can-div}

Let $D_0$ be a divisor on $V'$ with local equations $\Delta_0 = 0$
on $U_0$ and $\Delta_{\alpha, \beta} = 0$ on $U_{\alpha, \beta}$
for all $\alpha, \beta$ (cf. {\bf 2.1}). Note that the defining
equations of $D$, when lifted to the universal covers of
$U_0,U_1,\ldots$, are $(\mathbb{Z}/2\mathbb{Z})$-\emph{invariant}
(cf. Remark~\ref{remark:al-artin}). Then the sheaf property (see
\cite[Ch.~2]{artin}) implies that $D_0$ is a Cartier divisor on
$V'$. Let $\mathcal{L} := \mathcal{O}_{V'}(D_0)$ be the
corresponding line bundle.

\begin{lemma}
\label{theorem:l-metric} $D_0$ is irreducible and $\mathcal{L}$
carries a Hermitian metric $|\cdot|$ such that $1 = |\Delta_{0}| =
|\Delta_{\alpha, \beta}|$ on $U_0 \cap U_1$ and $U_{0} \cap
U_{\alpha, \beta}$ for all $\alpha, \beta$.
\end{lemma}

\begin{proof}
Evident.
\end{proof}

\begin{proposition}
\label{theorem:l-metric-prop} $D_0$ is ample.
\end{proposition}

\begin{proof}
Let $\theta \in H^{0}(V', \mathcal{L})$ be the global section such
that $(\theta)_0 = D_0$. Put $\theta_0 :=
\theta\big\vert_{\scriptscriptstyle U_{0}}$, $\theta_1 :=
\theta\big\vert_{\scriptscriptstyle U_{1}}$, $\theta_{\alpha,
\beta} := \theta\big\vert_{\scriptscriptstyle U_{\alpha, \beta}}$.

Restrict $\mathcal{L}$ to $U_0$ and define a Hermitian metric
$h_0$ on $\mathcal{L}\big\vert_{U_{0}}$ as follows:
$$
h_0 := (1 + |X_{1}|^{2})|\theta_0|.
$$
Then on $U_0 \cap U_1$ we have
$$
|\theta_1| = |\theta_0|\frac{1}{|\Delta_{0}|} = |\theta_0|,
$$
and hence
$$
h_0 = |\theta_1| + \frac{|X_{1}|^{2}}{|\Delta_{0}|^{2}}|\theta_1|
= (1 + |X^{(1)}_{1}|^{2})|\theta_1|.
$$
This extends $h_0$ to a metric on $\mathcal{L}$ over $U_0 \cup
U_1$. Repeating the same construction, with $U_1$ replaced by
$U_{\alpha, \beta}$, we obtain a global metric on $\mathcal{L}$,
equal
$$
(1 + |X^{(\alpha, \beta)}_{1}|^{2})|\theta_{\alpha, \beta}|
$$
on each $U_{\alpha, \beta}$. Moreover, starting with the metric
$$
h := |\theta_0|\prod_{i = 1}^{4}(1 + |X_{i}|^{2})^{1/4}
$$
on $\mathcal{L}$ over $U_0$, the same argument yields to a
metric\footnote{Equal $|\theta_{\alpha, \beta}|\prod_{i = 1}^{4}(1
+ |X^{(\alpha, \beta)}_{i}|^{2})^{1/4}$ on $U_{\alpha, \beta}$.}
on $\mathcal{L}$ over $X$ which extends $h$. Let us again denote
this new metric by $h$ and consider the $(1, 1)$-form $\Theta :=
\displaystyle\frac{\sqrt{-1}}{2\pi}\partial\bar{\partial}\log h
\in c_{1}(\mathcal{L})$. Then from the Nakai--Moishezon criterion
(see \cite[Th.~5.1]{hart}) we get the following:

\begin{lemma}
\label{theorem:if-theta-pos} If $\sqrt{-1}\Theta
> 0$, then $D_0$ is ample.
\end{lemma}

Further, the condition $\sqrt{-1}\Theta
> 0$ is local, so we restrict ourselves to the chart $U_0$
(the argument is the same for $U_1$ and $U_{\alpha, \beta}$), and
on $U_0$ we have
$$
\sqrt{-1}\Theta = \frac{1}{8\pi}\sum_{i = 1}^4 \frac{dX_i \wedge
d\bar{X}_{i}}{(1 + |X_{i}|^{2})^{2}} > 0.
$$
\end{proof}

\refstepcounter{equation}
\subsection{}
\label{subsection:can-div-ind}

There is a unique (prime) Cartier divisor $D_{\infty} \sim D_0$ on
$V'$ with equation $\Delta_1 = 0$ on $U_1$. Indeed, one can define
$D_{\infty}$ by taking the closure of the locus $(\Delta_1 = 0)
\subset U_1$ in $V'$, and $D_{\infty} \sim D_0$ because of the
rational map $V' \dashrightarrow \mathbb{P}^1$ which extends the
map $A \mapsto \det A$ on $U_0$. Equivalently, one can notice that
the divisors $D_{\infty}$ and $D_0 + (f)$ determine the same
valuations on the function field $\mathbb{C}(V')$, where $f$ is a
rational function on $V'$, equal $\Delta_{0}^{-1}$ on $U_0$ (cf.
Remark~\ref{remark:eqs-of-d-infty} below). Note also that $D_0 \ne
D_{\infty}$ (cf. the similar construction of $\mathbb{P}^1$ and of
the divisors $0, \infty \in \mathbb{P}^1$).

\begin{lemma}
\label{theorem:index-fano} $K_{V'} \sim -4D_0$.
\end{lemma}

\begin{proof}
Let us start with the form $\omega := dX_1 \wedge dX_2 \wedge dX_3
\wedge dX_4$ on $U_0$. We have
$$
\hat{X}_j := d\big(\frac{X_{j}}{X_{1}X_{4} - X_{2}X_{3}}\big) =
\frac{dX_{j}}{X_{1}X_{4} - X_{2}X_{3}} -
\frac{X_{j}d\big(X_{1}X_{4} - X_{2}X_{3}\big)}{\big(X_{1}X_{4} -
X_{2}X_{3}\big)^{2}}
$$
for all $j$, and it is easy to see that
\begin{equation}
\begin{array}{c}
\nonumber \hat{X}_1 \wedge \hat{X}_2 \wedge \hat{X}_3 \wedge
\hat{X}_4 =\displaystyle \frac{dX_{1} \wedge dX_{2} \wedge dX_{3}
\wedge dX_{4}}{\big(X_{1}X_{4} - X_{2}X_{3}\big)^{4}} -
\frac{\sum_{1 \leqslant j \leqslant 4}X_{j}d\big(X_{1}X_{4} -
X_{2}X_{3}\big)dX_{1} \wedge \ldots \wedge \hat{dX_{j}} \wedge
\ldots
\wedge dX_{4}}{\big(X_{1}X_{4} - X_{2}X_{3}\big)^{5}} = \\
\nonumber =\displaystyle \frac{dX_{1} \wedge dX_{2} \wedge dX_{3}
\wedge dX_{4}}{\big(X_{1}X_{4} - X_{2}X_{3}\big)^{4}}.
\end{array}
\end{equation}
Then we get
$$
dX_1 \wedge dX_2 \wedge dX_3 \wedge dX_4 =
\frac{1}{\Delta_{1}^{4}}~dX^{(1)}_{1} \wedge dX^{(1)}_{2} \wedge
dX^{(1)}_{3} \wedge dX^{(1)}_{4}
$$
on $U_0 \cap U_1$. This extends $\omega$ to a meromorphic form on
$U_0 \cup U_1$. Note that $K_{V'} = -4D_{\infty} \sim -4D_0$ on
$U_0 \cup U_1$.

Repeating the same construction, with $U_1$ replaced by
$U_{\alpha, \beta}$, we obtain a global meromorphic section of the
line bundle $\mathcal{O}_{V'}(K_{V'})$, equal
$$
\frac{1}{l_{\alpha, \beta}^{*}(\Delta_{\alpha,
\beta})^{4}}~dX^{(\alpha, \beta)}_{1} \wedge dX^{(\alpha,
\beta)}_{2} \wedge dX^{(\alpha, \beta)}_{3} \wedge dX^{(\alpha,
\beta)}_{4}
$$
on $U_0 \cap U_{\alpha, \beta}$ for all $\alpha, \beta$. Hence
$K_{V'} = -4D_{\infty} \sim -4D_0$ on $V'$.
\end{proof}

\begin{remark}
\label{remark:eqs-of-d-infty} It follows from the proof of
Lemma~\ref{theorem:index-fano} that the equation of the divisor
$D_{\infty}$ on $U_{\alpha, \beta}$ is $l_{\alpha,
\beta}^{*}(\Delta_{\alpha, \beta}) = 0$ for all $\alpha, \beta$.
\end{remark}

\begin{lemma}
\label{theorem:g-2-4-a-s-2} $V'$ is a quadratic cone with a unique
singular point.
\end{lemma}

\begin{proof}
Firstly, $V'$ has only isolated terminal singularities (by
definition of the latter and construction of $V'$). Now the
assertion follows from Lemma~\ref{theorem:index-fano},
Proposition~\ref{theorem:l-metric-prop} and \cite[Theorem
3.1.14]{isk-prok}.
\end{proof}

\bigskip

\section{Proof of Theorem~\ref{theorem:main}}
\label{section:quotient-2}

\refstepcounter{equation}
\subsection{}
\label{subsection:gr-action-v-prime}

Consider $V'$ as in Section~\ref{section:quotient-1}. Let us show
that the $G$-action extends from $U_0 =
\mathbb{C}^{4}/(\mathbb{Z}/2\mathbb{Z})$ to a regular action on
$V'$ (note $G$ is obviously defined on $U_0$).

By construction of $V'$, every $g \in G$ determines a birational
automorphism $g : V' \dashrightarrow V'$, regular and bijective on
$U_0 \cup U_1$. Furthermore, we have $V' \setminus{(U_0 \cup U_1)}
\subseteq D_0 \cup D_{\infty}$, since
$$
U_0 \cup U_1 \supseteq V' \setminus{(D_0 \cup D_{\infty})} = U_0
\cap U_1 \cap \bigcap_{\alpha, \beta} U_{\alpha, \beta}
$$
(cf. {\bf 2.1} and the equations of $D_0,~D_{\infty}$). Then,
since $g(D \cap U_{0}) = D \cap U_{0}$, $g(D_{\infty} \cap U_{1})
= D_{\infty} \cap U_{1}$ and $D_0,~D_{\infty}$ are irreducible, we
obtain that $g$ is an isomorphism in codimension $2$ on $V'$, and
hence $g_{*}(D) = D$, $g_{*}(D_{\infty}) = D_{\infty}$ in
$\mathrm{Pic}(V')$. This implies that $g$ is induced by an
automorphism of $\mathbb{P}^5 \supset V'$. Thus, we get $g \in
\mathrm{Aut}(V')$ and $U_{0}/G \approx V'/G$ (cf.
Lemma~\ref{theorem:g-2-4-a-s-1}).

\begin{remark}
\label{remark:prokhor} Note that given the embedding $U_0 :=
\mathbb{C}^4 \subset G(2, 4) =: V$, the $G$-action extends from
$U_0$ to $V$ by similar arguments as for $V'$ above. There is also
another construction (communicated by Yu.\,Prokhorov) of $V$ and
$G \subset \mathrm{Aut}(V)$ such that compactification $V \supset
U_0$ is $G$-equivariant. Indeed, take the standard
compactification of $U_0 := \mathbb{C}^{4}$ by $\mathbb{P}^4$,
with the divisor $B\subset\mathbb{P}^4$ at infinity, and extend
the $G$-action to $\mathbb{P}^4$ in the usual way. Then there is a
$G$-invariant smooth quadric $S \subset B = \mathbb{P}^{3}$. Let
$\sigma: Y \longrightarrow \mathbb{P}^4$ be the blow up of $S$
with the exceptional divisor $E := \sigma^{-1}(S)$. It is easy to
see that the linear system $|2L - E|$, $L := \sigma^{*}(B)$,
determines a birational contraction $\tilde{\sigma}: Y
\longrightarrow \tilde{Y}$, mapping the proper transform
$\sigma_{*}^{-1}(B) \sim L - E$ of the divisor $B$ to a point.
Moreover, since the normal bundle of $\sigma_{*}^{-1}(B) \simeq
\mathbb{P}^3$ on $Y$ is $\mathcal{O}_{\mathbb{P}^{3}}(-1)$, one
immediately gets that $\tilde{\sigma}$ is the blow up of a smooth
point on $\tilde{Y}$. Furthermore, $\tilde{Y}$ is a (smooth) Fano
$4$-fold, with $\mathrm{Pic}(\tilde{Y}) = \mathbb{Z}\cdot
\tilde{\sigma}_{*}(L)$ and such that
$\tilde{\sigma}^{*}(K_{\tilde{Y}}) = K_Y - 3\sigma_{*}^{-1}(B) =
-4L$, i.e., the Fano index of $\tilde{Y}$ is $4$. Hence, by
\cite[Theorem 3.1.14]{isk-prok}, $\tilde{Y}$ is a smooth quadric
in $\mathbb{P}^5$. Finally, the construction of $\tilde{Y}$
implies that both $\sigma$ and $\tilde{\sigma}$ are
$G$-equivariant. Hence $\tilde{Y}$ ($=V$) is a $G$-equivariant
compactification of $U_0$. However, we could not obtain similar
(``Italian") construction for $V'$, since the way we have built
$V'$ is not actually birational. Yet we need $V'$ to have, for
instance, such properties as Lemma~\ref{theorem:many-invariants}
below (which does not hold for the smooth quadric $V$).
\end{remark}

\begin{lemma}
\label{theorem:many-invariants} The space $H^{0}(V',
\mathcal{O}_{V'}(D_0))$ contains three linearly independent
$G$-invariant elements.
\end{lemma}

\begin{proof}
Note that $D_0$ and $D_{\infty}$ are $G$-invariant. Moreover,
since $D_0$ and $D_{\infty}$ are hyperplane sections of $V'
\subset \mathbb{P}^5$ which pass through the vertex $O\in
V'$,\footnote{Indeed, we have $D_0 \cap U_0 = (X_{1}X_{4} -
X_{2}X_{3} = 0)$, hence $O \in D_0$, and similarly for
$D_{\infty}$ on $U_1$.} there is also a smooth $G$-invariant
hyperplane section $H$ of $V'$. Indeed, consider the linear
projection $V' \dashrightarrow Q$ from $O$, with $Q \subset
\mathbb{P}^4$ being a smooth quadric (cf.
Lemma~\ref{theorem:g-2-4-a-s-2}). Let also $f: V'' \longrightarrow
V'$ be the blow up of $O$. Then we get $V'' =
\mathbb{P}(\mathcal{E})$ for some $\mathbb{C}^2$-vector bundle
$\mathcal{E}$ over $Q$ such that the natural projection $V''
\longrightarrow Q$ is $G$-equivariant.

Further, since both $\mathbb{O}, \mathbb{I}\subset G$ are simple
and commute with $\mathbb{C}^*$, the class of $\mathcal{E}$ in
$H^1(Q, GL_2(\mathcal{O}_Q))$ is $G$-invariant. Hence the
$G$-action on $V''$ extends to the one on $\mathcal{E}$. Now,
$\mathcal{E}$ admits two $G$-invariant sections, the $0$-section
and the one corresponding to the exceptional divisor of $f$. This
implies that the $G$-action on the fibers of the projection $V''
\longrightarrow Q$ coincides with the $\mathbb{C}^*$-action. The
existence of the above $H$ is now evident.

Finally, $D_0, D_{\infty}$ and $H$ are (obviously) linearly
independent in $H^{0}(V', \mathcal{O}_{V'}(D_0))$.
\end{proof}

\begin{lemma}
\label{theorem:v-prime} The $3$-fold $V'/G$ is rational.
\end{lemma}

\begin{proof}
By Lemma~\ref{theorem:many-invariants}, we may assume the equation
of $V' \subset \mathbb{P}^5 = \mathrm{Proj}\big(\mathbb{C}[x_{0},
\ldots, x_{5}]\big)$ to be $x_{0}x_{1} + x_{2}x_{3} + x_{4}^2 =
0$, with $\mathbb{C}^* \subset G$ acting diagonally and
$\mathbb{O} \times \mathbb{I} \subset G$ fixing $x_0$, $x_1$,
$x_5$. Let $V' \dashrightarrow \mathbb{P}^2$ be the restriction to
$V'$ of the linear projection from the $G$-invariant plane $\Pi :=
(x_2 = x_3 = x_4 = 0)$. Note that $V' \cap \Pi$ is a pair of
distinct lines (with trivial $\mathbb{O} \times
\mathbb{I}$-action). Then, blowing up $V'$ at $V' \cap \Pi$, we
get a normal $4$-fold $V'' \approx_G V'$ together with a
$G$-equivariant morphism $V'' \longrightarrow \mathbb{P}^2$ which
has at least three $G$-invariant sections and generic fiber
$\approx [\text{a quadratic cone}]$. In particular, we get
$$
V' \approx_G \left[\text{quadratic cone with trivial ($\mathbb{O}
\times \mathbb{I}$)-action}\right] \times \mathbb{P}^2,
$$
which implies that $V'/G$ is rational.
\end{proof}

\begin{lemma}
\label{theorem:v-prime-1} The $3$-fold $V/G$ is rational.
\end{lemma}

\begin{proof}
We have
$$
\mathbb{C}^4/G = \mathbb{C}^4/(\mathbb{O} \times \mathbb{I} \times
\mathbb{C}^{*}) \simeq \mathbb{C}^4/(\mathbb{O} \times \mathbb{I}
\times \mathbb{C}^{*} \times \mathbb{Z}/2\mathbb{Z}) = U_{0}/G
$$
for the (non-canonical) isomorphism $\mathbb{C}^{*} \simeq
\mathbb{C}^{*}/(\mathbb{Z}/2\mathbb{Z})$. Now the statement
follows from Lemma~\ref{theorem:v-prime} because $\mathbb{C}(V'/G)
= \mathbb{C}(U_0/G)$.
\end{proof}

Lemma~\ref{theorem:v-prime-1} proves Theorem~\ref{theorem:main}.

\bigskip

\thanks{{\bf Acknowledgments.} I would like to thank I.\,Cheltsov, S.\,Galkin
and Yu.\,G.\,Prokhorov for valuable remarks and exceptional
patience during the preparation of this paper. Also the referee's
remarks have allowed to improve the exposition.


\bigskip

\begin{thebibliography}{4}

\bibitem{artin}
Artin M. Algebraic spaces // Yale Univ. Press, New Haven, CT. 1971.

\smallskip

\bibitem{hart}
Hartshorne R. Ample subvarieties of algebraic varieties // Lecture
Notes in Mathematics. Vol. 156. Berline: Springer Verlag. 1970.

\smallskip

\bibitem{isk-prok}
Iskovskikh V. A., Prokhorov Yu. G. Fano varieties. Encyclopaedia
of Mathematical Sciences // Algebraic geometry V / ed. Parshin A.
N., Shafarevich I. R. V. 47. Berlin: Springer Verlag. 1999.

\smallskip

\bibitem{prok-inv}
Prokhorov Yu. G. Fields of invariants of finite linear groups //
Cohomological and geometric approaches to rationality problems. V.
282 of Progr. Math. P. 245--273. Birkh\"auser Boston Inc., Boston,
MA. 2010.

\end{thebibliography}
\end{document}